\documentclass{amsart}
\usepackage{amsmath}
\usepackage{amssymb}
\usepackage{amsthm}
\usepackage{enumerate}
\usepackage[dvips]{graphicx}
\usepackage{pb-diagram}
\theoremstyle{definition}
\newtheorem{definition}{Definition}[section]
\theoremstyle{plain}
\newtheorem{lemma}[definition]{Lemma}
\newtheorem{theorem}[definition]{Theorem}

\newtheorem{proposition}[definition]{Proposition}
\newtheorem{corollary}[definition]{Corollary}
\theoremstyle{remark}
\newtheorem{remark}[definition]{Remark}

\makeatletter
\@namedef{subjclassname@2020}{
  \textup{2020} Mathematics Subject Classification}
\makeatother

\begin{document}

\title[Simple product]{Simple product and locally o-minimal theories}
\author[M. Fujita]{Masato Fujita}
\address{Department of Liberal Arts,
Japan Coast Guard Academy,
5-1 Wakaba-cho, Kure, Hiroshima 737-8512, Japan}
\email{fujita.masato.p34@kyoto-u.jp}

\begin{abstract}
There exist NIP and non-NTP$_2$ theories satisfying the following conditions:
It is not o-minimal;
All models are strongly locally o-minimal; 
It has a model which is an expansion of the linearly ordered abelian group over the reals $(\mathbb R\;;\;0,+,<)$.
We construct these examples using the notion of simple product.
\end{abstract}



\maketitle

\section{Introduction}\label{sec:intro}
Under Shelah's classification program and the works influenced by this program, theories having or not having different combinatorial configurations are introduced and investigated, such as stable theories which are found in standard texts of model theories, simple theories \cite{KP}, dp-minimal theories \cite{S2,S3}, NIP theories \cite{S} and NTP$_2$ theories \cite{Sh,C,YC}.
O-minimal theories \cite{vdD,KPS,PS} and their relatives \cite{DMS,MMS,TV} have different taste.
They are extensions of linear orders.
They often treat structures rather than theories, and they consider definable sets rather than types.

We can classify o-minimal theories and their relatives from the view point of combinatorial configurations in several cases.
For instance, o-minimal theories and their relatives are not simple in many case because the theory of dense linear order without endpoints is not simple by the definition of simple theories. 
A weakly o-minimal structure is dp-minimal \cite[A.1.3]{S} and it is NIP.
Dense pairs of o-minimal structures \cite{vdD2} are NIP by \cite{CS}.
NIP structures having o-minimal open cores are constructed in \cite{HNW}.
NIP expansions of the ordered group of reals are studied in \cite{W2}.

This paper treats locally o-minimal theories, which is a relative of o-minimal theories.
The notion of local o-minimality was first studied in \cite{TV} and has been studied in \cite{F,KTTT,Fuji,Fuji2,Fuji3,Fuji4}.
A non-NIP locally o-minimal theory is already constructed in \cite[Example 5.17]{F}.
The example is the ultraproduct of o-minimal theories.
In this paper, we construct much more well-behaved locally o-minimal theory which is NIP and not NTP$_2$.
Recall that non-NTP$_2$ implies non-NIP.
We construct NIP and non-NTP$_2$ locally o-minimal theories satisfying the following conditions:
\begin{itemize}
\item It is not o-minimal;
\item All models are strongly locally o-minimal; 
\item It has a model which is an expansion of the linearly ordered abelian group over the reals $(\mathbb R\;;\;0,+,<)$.
\end{itemize}
Locally o-minimal expansions of the linearly ordered abelian group over the reals $(\mathbb R\;;\;0,+,<)$ enjoy tame topological properties unavailable in general locally o-minimal structures by \cite{Fuji5}.
The above example indicates that the both notions of NIP and NTP$_2$ fail to classify these locally o-minimal theories enjoying very good topological properties.
It contrasts that dp-minimality successfully classifies o-minimal theories in definably complete extensions of ordered groups (Proposition \ref{prop:dp-minimal}).

When we construct the examples, we use the notion of simple product introduced in \cite{KTTT}.
It is used in \cite{W} so as to investigate the structure $(\mathcal R,\alpha\mathbb Z)$, which is a generalization of \cite{BC}.
Here, $\mathcal R$ is an o-minimal expansion of the ordered group of reals.
Simple product is originally defined for structures but not for theories.
We develop the notion of the simple products of theories in this paper.
The theory of the simple product inherits the properties of the original theories.
It is a very good tool to construct locally o-minimal examples.

This brief note consists of four sections.
Section \ref{sec:simple} is the main body of this note.
We give the definition of simple product here.
We demonstrate that the simple product of two theories are NIP if and only if the original two theories are NIP in Section \ref{sec:nip}.
We also show a similar assertion for distal theories.
Finally, in Section \ref{sec:locally}, we construct the theories introduced in this section.

\section{Modeling simple product}\label{sec:simple}
We first formalize the Cartesian product of sets in a model theory.
We employ one-sort formalization of the  product in this paper.

\begin{definition}[Theory of Cartesian product]\label{def:cartesian}
Let $L_{\times}$ be the language consisting of two binary predicates $\sim_1$ and $\sim_2$.
The theory $T_{\times}$ consists of the following sentences:
\begin{enumerate}
\item[(1)] The relations $\sim_1$ and $\sim_2$ are equivalence relations;
\item[(2)] $(\forall x) (\forall y) ((x \sim_1 y) \wedge (x \sim_2 y) \rightarrow x=y)$;
\item[(3)] $(\forall x) (\forall y) (\exists z) (x \sim_1 z) \wedge (y \sim_2 z)$.
\end{enumerate}
For all tuples $\overline{x}=(x_1,\ldots, x_n)$ and $\overline{y}=(y_1,\ldots, y_n)$, the notation $\overline{x} \sim_1 \overline{y}$ denotes the formula $\bigwedge_{i=1}^n (x_i \sim_1 y_i)$.
We define $\overline{x} \sim_2 \overline{y}$, similarly.
\end{definition}

\begin{proposition}\label{prop:cartesian}
Let $S_1$ and $S_2$ be nonempty sets.
Let $\pi_i:S_1 \times S_2 \rightarrow S_k$ be the projection onto the $k$-th coordinate for $k=1,2$.
We set 
$$x \sim_k^{\mathcal M_{\times}} y \Leftrightarrow \pi_k(x)=\pi_k(y)\text{,}$$
where $x$ and $y$ are elements in $S_1 \times S_2$.
The tuple $\mathcal M_{\times}=(S_1 \times S_2;\sim_1^{\mathcal M_{\times}},\sim_2^{\mathcal M_{\times}})$ is a model of $T_{\times}$.

Conversely, any model $\mathcal M=(S;\sim_1^{\mathcal M},\sim_2^{\mathcal M})$ of $T_{\times}$ is isomorphic to the Cartesian product model $\mathcal M_{\times}=(S_1 \times S_2;\sim_1^{\mathcal M_{\times}},\sim_2^{\mathcal M_{\times}})$ of some sets $S_1$ and $S_2$.
\end{proposition}
\begin{proof}
It is obvious that the tuple $\mathcal M_{\times}=(S_1 \times S_2;\sim_1^{\mathcal M_{\times}},\sim_2^{\mathcal M_{\times}})$ is a model of $T_{\times}$.

We demonstrate the converse.
Since the relation $\sim_k^{\mathcal M}$ is an equivalence relation by Definition \ref{def:cartesian}(1), the quotient space $S_k=S/\sim_k^{\mathcal M}$ is well-defined for $i=1,2$.
For any $x \in S$, let $[x]_k$ be the equivalence class of $x$ under the equivalence relation $\sim_k$.
The map $\sigma:S \rightarrow S_1 \times S_2$ is defined by $\sigma(x)=([x]_1,[x]_2)$.

The map $\sigma$ is an isomorphism.
The bijectivity of $\sigma$ follows from Definition \ref{def:cartesian}(2) and (3).
For any $x,y \in S$, the equivalence of the condition that $x \sim_k y$ with the condition that $\pi_k(\sigma(x))=\pi_k(\sigma(y))$ is also obvious from the definition of the map $\sigma$.
\end{proof}

The simple product of two structures are defined in \cite[Section 4]{KTTT}.
Let $L_1$ and $L_2$ be two languages, and $T_1$ and $T_2$ be their theories.
Our purpose is to construct a theory $T_{\text{sim}}$ such that any model of $T_{\text{sim}}$ is a simple product.
We first review the notion of simple products of two structures.

\begin{definition}\cite[Definition 14]{KTTT}
Let $L_1$ and $L_2$ be two languages.
Consider $L_k$-structure $\mathcal M_k$ and its universe $M_k$ for $k=1,2$.
Set $N=M_1 \times M_2$.
For subsets $A \subset M_1^n$ and $B \subset M_2^n$, we set 
$$A*B=\{((a_1,b_1),\ldots, (a_n,b_n)) \in N^n\;|\; (a_1, \ldots, a_n) \in A,\ (b_1, \ldots, b_n) \in B\}\text{.}$$

Consider another language $L$.
An $L$-structure $\mathcal N$ is called a \textit{simple product} of $\mathcal M_1$ and $\mathcal M_2$ if the universe $\mathcal N$ is the Cartesian product $N=M_1 \times M_2$ and any $N$-definable set in $N^n$ is a boolean combination of the following sets:
\begin{itemize}
\item $A_i*M_1^n\ (1 \leq i \leq k)$ and 
\item $M_1^n*B_j\ (1 \leq j \leq l)$,
\end{itemize}
where $A_i$ are $\mathcal M_1$-definable subsets of $M_1^n$ and $B_j$ are $\mathcal M_2$-definable subsets of $M_2^n$.
\end{definition}

In the original definition, a simple product of two structures is not uniquely determined.
Therefore, we define the standard simple product which is uniquely determined from the given structures.
To this end, we first define the language of simple product of two languages.

\begin{definition}[Language of simple product]\label{def:simple_lan}
Let $L_1$ and $L_2$ be two languages.
We assume that $L_1$ and $L_2$ do not contain function symbols.
When they have function symbols, we convert them to predicate symbols by replacing them with their graphs.
We consider that $L_1$ and $L_2$ have empty intersection.
It means that, even if a symbol is commonly used in both $L_1$ and $L_2$, we distinguish them.
We also assume that both languages contain at least one constant symbol.

We define the language $L_{\text{sim}}$.
For any pair of constant symbols $(c_1,c_2)$ with $c_1 \in L_1$ and $c_2 \in L_2$, there exists a unique constant symbol $C_{(c_1,c_2)}$ in $L_{\text{sim}}$.
The language $L_{\text{sim}}$ does not have other constant symbols.

The set of predicate symbols in $L_{\text{sim}}$ is the disjoint union of those in three languages $L_1$, $L_2$ and $L_{\times}$.
Recall that $L_{\times}$ is the language defined in Definition \ref{def:cartesian}.
The language $L_{\text{sim}}$ is called the \textit{language of simple product of $L_1$ and $L_2$}.
\end{definition}

We then give the definition of the standard simple product. 

\begin{definition}[Standard simple product]\label{def:standard_sp}
Let $L_1$ and $L_2$ be two languages which do not have a function symbol.
Consider $L_k$-structures $\mathcal M_k$ and their universes $M_k$ for $k=1,2$.
Set $N=M_1 \times M_2$.

We define an $L_{\text{sim}}$-structure $\mathcal N_{\text{st}}$ as follows.
The universe of $\mathcal N_{\text{st}}$ is the Cartesian product $N=M_1 \times M_2$.
Let $\pi_k:M_1 \times M_2 \rightarrow M_k$ be the projection onto the $k$-th coordinate for $k=1,2$.
The interpretation $C_{(c_1,c_2)}^{\mathcal N_{\text{st}}}$ of the constant symbol $C_{(c_1,c_2)}$ is $(c_1^{\mathcal M_1}, c_2^{\mathcal M_2}) \in N$, where $c_k^{\mathcal M_k}$ is the interpretation of $c_k$ in $\mathcal M_k$ for $k=1,2$.
For any $x,y \in N$, we define $\mathcal N_{\text{st}} \models (x \sim_k y)$ by the condition $\pi_k(x)=\pi_k(y)$ for $k=1,2$.

For any tuple $\overline{c}=(c_1, \ldots, c_n) \in N^n$, let $\pi_k(\overline{c})$ denote the tuple $(\pi_k(c_1), \ldots, \pi_k(c_n))$.
Let $R$ be an $n$-ary predicate symbol in $L_k$.
We set $\mathcal N_{\text{st}} \models R(\overline{c})$ by $\mathcal M_k \models R(\pi_k(\overline{c}))$.

The $L_{\text{sim}}$-structure $\mathcal N_{\text{st}}$ is called the \textit{standard simple product} of $\mathcal M_1$ and $\mathcal M_2$.
\end{definition}

\begin{definition}[Theory of simple product]\label{def:simple_theory}
Let $L_1$ and $L_2$ be two languages given in Definition \ref{def:simple_lan}.
Fix $k=1,2$.
Let $\phi(\overline{x})$ be an $L_k$-formula.
We construct an $L_{\text{sim}}$-formula $\widetilde{\phi}(\overline{x})$ from the formula $\phi(\overline{x})$ by induction on the complexity of $\phi(\overline{x})$.
The formula $\widetilde{\phi}(\overline{x})$ is called the \textit{standard conversion} of the formula $\phi(\overline{x})$.
\begin{enumerate}
\item[(a)] Let $x$ and $y$ be two $L_k$-terms.
Since $L_k$ does not have a function symbol, $x$ and $y$ are either constant symbols or variables.
Consider the case in which the formula $\phi(\overline{x})$ is of the form $x=y$.
\begin{itemize}
\item When both $x$ and $y$ are variables, the standard conversion $\widetilde{\phi}(x,y)$ is the formula $x \sim_k y$.
\item When $x$ is a variable and $y$ is a constant symbol and $k=1$, take an arbitrary constant symbol $c$ in $L_2$.
The standard conversion $\widetilde{\phi}(x)$ is the formula $x \sim_1 C_{y,c}$.
We define the standard conversion similarly in the other cases in which only one of $x$ and $y$ is a constant symbol.
\item When both $x$ and $y$ are constant symbols and $k=1$, take arbitrary constant symbols $c_1$ and $c_2$ in $L_2$, and $\widetilde{\phi}$ is the sentence of the form $(C_{x,c_1} \sim_1 C_{y,c_2})$.
The case in which $k=2$ is similar.
\end{itemize}
\item[(b)] Let $\overline{x}$ be a tuple of variables and $R$ be a predicate symbol in $L_k$.
When the formula $\phi$ is of the form $R(\overline{x})$, the standard conversion $\widetilde{\phi}$ is also the formula $R(\overline{x})$.
\item[(c)] When the formula $\phi(\overline{x})$ is the conjunction $\psi_1(\overline{x}) \wedge \psi_2(\overline{x})$ and $\widetilde{\psi_i}(\overline{x})$ are the standard conversions of $\psi_i(\overline{x})$ for all $i=1,2$.
The standard conversion $\widetilde{\phi}(\overline{x})$ of $\phi(\overline{x})$ is given by $\widetilde{\psi_1}(\overline{x}) \wedge \widetilde{\psi_2}(\overline{x})$.
The standard conversion of $\phi'(\overline{x})=\neg \psi_1(\overline{x})$ is given by $\widetilde{\phi'}(\overline{x})=\neg \widetilde{\psi_1}(\overline{x})$.
\item[(d)] When the formula $\phi(\overline{x})$ is of the form $(\exists y) \psi(\overline{x},y)$ and $\widetilde{\psi}(\overline{x},y)$ are the standard conversions of $\psi(\overline{x},y)$.
The standard conversion $\widetilde{\phi}(\overline{x})$ of $\phi(\overline{x})$ is given by $(\exists y) \widetilde{\psi}(\overline{x},y)$.
\item[(e)] Let $\overline{x}=(x_1,\ldots, x_n)$ be an $L_k$-terms and contains a constant symbol and $R$ be a predicate symbol in $L_k$.
The formula $R(\overline{x})$, it is equivalent to the formula
$$(\exists \overline{y}) \left(R(\overline{y}) \wedge \bigwedge_{i=1}^n (x_i=y_i)\right)\text{,}$$
where $\overline{y}=(y_1,\ldots,y_n)$ is a tuple of variables.
Apply (a) through (d) to the obtained formula.
The output is the standard conversion of $R(\overline{x})$.
\end{enumerate}

Let $T_k$ be theories of $L_k$ for $k=1,2$.
The set of $L_{\text{sim}}$-sentences $\widetilde{T_k}$ is given by 
$$\widetilde{T_k}=\{\widetilde{\sigma}\;|\; \sigma \in T_k\}\text{,}$$
where $\widetilde{\sigma}$ is the standard conversion of an $L_k$-sentence $\sigma$.
The set of $L_{\text{sim}}$-sentences $T_{\text{sim}}$ consists of the following sentences:
\begin{enumerate}
\item[(1)] Sentences in $T_{\times}$, $\widetilde{T_1}$ and $\widetilde{T_2}$;
\item[(2)] For all $k=1,2$ and all the predicate symbols $R$ in $L_k$, $(\forall \overline{x}) (\forall \overline{y})( (\overline{x} \sim_k \overline{y}) \rightarrow (R(\overline{x}) \leftrightarrow R(\overline{y})))$.
\end{enumerate}
The theory $T_{\text{sim}}$ is called the \textit{theory of simple products} of $T_1$ and $T_2$.
\end{definition}

Any model of $T_{\text{sim}}$ is isomorphic to the standard simple product of two models of $T_1$ and $T_2$.
We begin to prove this claim.

\begin{lemma}\label{lem:standard}
Let $L_1$ and $L_2$ be two disjoint languages which do not have a function symbol and contain at least one constant symbol.
Let $\mathcal M_k=(M_k,\ldots)$ be $L_k$-structures for $k=1,2$.
Consider an $L_{\text{sim}}$-structure $\mathcal N$ which is simultaneously a model of $T_{\times}$ and isomorphic to the Cartesian product of $M_1$ with $M_2$ under the isomorphism $\sigma$ as a model of $T_{\times}$.
Let $\pi_k:M_1 \times M_2 \rightarrow M_k$ be the projection onto the $k$-th component for $k=1,2$.
Assume further that the condition that $x \sim_k y$ is equivalent to the condition that $\pi_k(\sigma(x))=\pi_k(\sigma(y))$ and the sentences given in Definition \ref{def:simple_theory}(2) are all satisfied.

Fix $k=1,2$.
Let $\phi(\overline{x})$ be an $L_k$-formula and $\widetilde{\phi}(\overline{x})$ be its standard conversion.
We have $\mathcal M_k \models \phi(\pi_k(\overline{x}))$ if and only if $\mathcal N \models \widetilde{\phi}(\sigma(\overline{x}))$ for any $\overline{x} \in (M_1 \times M_2)^n$.
\end{lemma}
\begin{proof}
We prove it by induction on the complexity of the formula $\phi(\overline{x})$.
The lemma is almost immediate from Definition \ref{def:simple_theory}(a) through (d).
\end{proof}

\begin{theorem}\label{thm:standard_sp}
Let $L_1$ and $L_2$ be two disjoint languages which do not have a function symbol and contain at least one constant symbol.
Let $T_1$ and $T_2$ be their theories.
Consider models $\mathcal M_k$ of $T_k$ for $k=1,2$.
The standard simple product of $\mathcal M_1$ and $\mathcal M_2$ is a model of the theory $T_{\text{sim}}$ of simple products of $T_1$ and $T_2$.

Conversely, any model of $T_{\text{sim}}$ is isomorphic to the standard simple product of two models of $T_1$ and $T_2$.
\end{theorem}
\begin{proof}
We first demonstrate that the simple product $\mathcal N=(M_1 \times M_2,\ldots)$ of $\mathcal M_1=(M_1, \ldots)$ and $\mathcal M_2=(M_2, \ldots)$ is a model of $T_{\text{sim}}$.
It is obvious that the sentences in Definition \ref{def:simple_theory}(2) are satisfied by the definition of standard simple product.
Fix $k=1,2$.
The sentences in $\widetilde{T_k}$ are all satisfied in $\mathcal N$ by Lemma \ref{lem:standard} because $\mathcal M_k$ are models of $T_k$.

We next show the converse part.
Fix a model $\mathcal N=(N,\ldots)$ of $T_{\text{sim}}$.
Set $M_k=N/\sim_k^{\mathcal N}$ and $\sigma: N \rightarrow M_1 \times M_2$ be the $L_{\times}$-isomorphism defined in the proof of Proposition \ref{prop:cartesian}.
We define $L_k$ structures $\mathcal M_k$ as follows for $k=1,2$.
The universe of $\mathcal M_k$ is the set $M_k$.
For any predicate $R$ in $L_k$ we define $\mathcal M_k \models R(\overline{x})$ for $\overline{x} \in M_k^n$ by the condition that $\mathcal N \models R(\overline{y})$, where $\overline{y}$ is a representative in $N$ of the residue class $\overline{x}$.
The definition does not depend on the choice of representatives by Definition \ref{def:simple_theory}(2).
Let $\pi_k:M_1 \times M_2 \rightarrow M_k$ be the projection onto the $k$-th component for $k=1,2$.
It is obvious that 
$$x \sim_k y \Leftrightarrow \pi_k(\sigma(x))=\pi_k(\sigma(y))$$
for any $x,y \in N$.
Therefore, the map $\sigma$ induces the isomorphism between the model $\mathcal N$ and the standard simple product of $\mathcal M_1$ and $\mathcal M_2$.

The remaining task is to demonstrate that $\mathcal M_k$ are models of $T_k$ for $k=1,2$.
It immediately follows from Lemma \ref{lem:standard} and Definition \ref{def:simple_theory}(1).
\end{proof}

The next task is to demonstrate the standard simple product is truly a simple product.
\begin{theorem}\label{thm:standard_sp2}
Let $L_1$ and $L_2$ be two disjoint languages which do not have a function symbol and contain at least one constant symbol.
Let $T_1$ and $T_2$ be their theories.
Let $T_{\text{sim}}$ be the theory of simple products of $T_1$ and $T_2$.
For any $L_{\text{sim}}$-formula $\phi(\overline{x})$, there are finitely many $L_1$-formulas $\psi_{11}(\overline{x}), \ldots \psi_{1N}(\overline{x})$ and $L_2$-formulas $\psi_{21}(\overline{x}), \ldots \psi_{2N}(\overline{x})$ such that
\[
T_{\text{sim}} \models (\forall \overline{x}) \left( \phi(\overline{x}) \leftrightarrow \bigvee_{i=1}^N \widetilde{\psi}_{1i}(\overline{x}) \wedge \widetilde{\psi}_{2i}(\overline{x}) \right)\text{,}
\] 
where $\widetilde{\psi}_{ij}(\overline{x})$ is the standard conversion of $\psi_{ij}(\overline{x})$ for $i=1,2$ and $1 \leq j \leq N$.
\end{theorem}
\begin{proof}
We first introduce several terms.
An $L_{\text{sim}}$-formula is called \textit{simple} if it is of the form $\widetilde{\psi}_{1}(\overline{x}) \wedge \widetilde{\psi}_{2}(\overline{x})$, where $\widetilde{\psi}_{k}(\overline{x})$ are the standard conversions of $L_k$-formulas $\psi_k$ for all $k=1,2$.
An $L_{\text{sim}}$-formula is called \textit{semi-simple} if it is the disjunction of finitely many simple formulas.
The theorem says that any $L_{\text{sim}}$-formula is equivalent to a semi-simple formula.
We prove the theorem by induction on the complexity of the $L_{\text{sim}}$-formula $\phi(\overline{x})$.

We first consider the case in which the formula $\phi(\overline{x})$ is an atomic formula of the form $x=y$, where $x$ and $y$ are $L_{\text{sim}}$ terms.
We have $$T_{\text{sim}} \models (\forall x) (\forall y) ((x=y) \leftrightarrow ((x \sim_1 y) \wedge (x \sim_2 y))$$ by Definition \ref{def:cartesian}(1) and (2).
The formula $x \sim_k y$ is the standard conversion of the $L_k$-formula for $k=1,2$.
In fact, when both $x$ and $y$ are variables, $x \sim_k y$ is the standard conversion of the formula $x=y$.
When only one of $x$ and $y$, say $x$, is a variable, we have $y=C_{c_1,c_2}$ for some constant symbol.
The formula $x \sim_k y$ is the standard conversion of the formula $x=c_k$.
Finally, when both $x$ and $y$ are constant symbols, we have $x=C_{c_1,c_2}$ and $y=C_{d_1,d_2}$, where $c_k$ and $d_k$ are in $L_k$ for $k=1,2$.
The formula $x \sim_k y$ is the standard conversion of the formula $c_k=d_k$.
Consequently, the formula $\phi(x,y)$ is equivalent to a simple formula.

We next consider the case in which $\phi(\overline{x})$ is of the form $R(\overline{x})$, where $R$ is a predicate symbol in $L_1$ or $L_2$ and $\overline{x}$ is a tuple of $L_{\text{sim}}$-terms.
In the same manner as Definition \ref{def:simple_theory}(e), we can reduce to the case in which $\overline{x}$ is a tuple of variables.
The $R(\overline{x})$ is simple by the definition.

The conjunction of two semi-simple formulas and the negation of a semi-simple formula are obviously semi-simple by Definition \ref{def:simple_theory}(c).

We finally consider the case in which $\phi(\overline{x})$ is of the form $(\exists y) \psi(\overline{x},y)$, where $\psi(\overline{x},y)$ is a semi-simple formula.
We may assume that $\psi$ is simple without loss of generality.
There exist an $L_1$-formula $\psi_1(\overline{x},y)$ and an $L_2$-formula $\psi_2(\overline{x},y)$ such that $\psi(\overline{x},y) = \widetilde{\psi_1}(\overline{x},y) \wedge \widetilde{\psi_2}(\overline{x},y)$.
Here, $\widetilde{\psi}_k(\overline{x},y)$ denotes the standard conversion of the $L_k$-formula $\psi_k(\overline{x},y)$ for $k=1,2$.
We consider the formula $\phi'(\overline{x})$ given by 
$$((\exists y_1) \widetilde{\psi_1}(\overline{x},y_1)) \wedge ((\exists y_2) \widetilde{\psi_2}(\overline{x},y_2)) \text{.}$$
The formula $((\exists y_k) \widetilde{\psi_k}(\overline{x},y_k))$ is the standard conversion of the $L_k$ formula $$((\exists y_k) \psi_k(\overline{x},y_k))$$ by Definition \ref{def:simple_theory}(d) for any $k=1,2$.
The formula $\phi'(\overline{x})$ is simple.
Therefore, we have only to show that the formula $\phi(\overline{x})$ is equivalent to the formula $\phi'(\overline{x})$.
Obviously, the formula $\phi(\overline{x})$ implies the formula $\phi'(\overline{x})$.
We demonstrate the opposite implication.

Take an arbitrary model $\mathcal M=(M,\ldots)$ of $T_{\text{sim}}$.
Assume that $\overline{c} \in M^n$ satisfies the formula $\phi'(\overline{x})$.
We can take $d_k \in M$ such that $\mathcal M \models \widetilde{\psi_k}(\overline{c},d_k)$ for $k=1,2$.
We can take $d \in M$ with $d \sim_1 d_1$ and $d \sim_2 d_2$ by Definition \ref{def:cartesian}(3).
We have $\mathcal M \models \widetilde{\psi_k}(\overline{c},d)$ for $k=1,2$ by Lemma \ref{lem:standard}.
It means that $\mathcal M \models \phi(\overline{c})$.
We have finished the proof.
\end{proof}

\begin{corollary}\label{cor:standard_sp}
Let $L_1$ and $L_2$ be two disjoint languages which do not have a function symbol.
Consider $L_k$-structures $\mathcal M_k$ for $k=1,2$.
Their standard simple product $\mathcal N_{\text{st}}$ is a simple product.
\end{corollary}
\begin{proof}
Immediate from Lemma \ref{lem:standard}, Theorem \ref{thm:standard_sp} and Theorem \ref{thm:standard_sp2}.
\end{proof}

\section{NIP and distal theories}\label{sec:nip}
We now demonstrate that the theory of simple products of the theories $T_1$ and $T_2$ is NIP if and only if both $T_1$ and $T_2$ are NIP.
We first review the definition of NIP theory.

\begin{definition}\cite{S}
Let $T$ be a complete theory with infinite model and $L$ be its language.
An $L$-formula $\phi(\overline{x};\overline{y})$ has IP (the independence property) if
 there are $\{\overline{a}_i\;|\;i < \omega\}$ and $\{\overline{b}_I\;|\; I \subset \omega\}$ such that
$$\models \phi(\overline{a}_i;\overline{b}_I) \Leftrightarrow i \in I \text{.}$$
By compactness theorem, $\phi(\overline{x};\overline{y})$ has IP if and only if there are $\{\overline{a}_i\;|\;i < \omega\}$ and $\{\overline{b}_I\;|\; I \subset \omega\text{: finite}\}$ such that
$$\models \phi(\overline{a}_i;\overline{b}_I) \Leftrightarrow i \in I \text{.}$$
If $\phi$ does not have IP, we say it has NIP.
T has IP if some formula has IP in T, and otherwise it is said that $T$ has NIP.
\end{definition}

We now get the following theorem:
\begin{theorem}\label{thm:inherit}
Let $L_1$ and $L_2$ be two disjoint languages which do not have a function symbol and contains at least one constant symbol.
Let $T_1$ and $T_2$ be their theories.
Let $T_{\text{sim}}$ be the theory of simple product of $T_1$ and $T_2$.
The theory $T_{\text{sim}}$ is NIP if and only if both the theories $T_1$ and $T_2$ are NIP.
\end{theorem}
\begin{proof}
Assume that one of $T_1$ and $T_2$ is not NIP.
We assume that $T_1$ is not NIP without loss of generality.
Let $\phi(\overline{x};\overline{y})$ be an $L_1$-formula having IP.
There exists a model $\mathcal M_1=(M_1,\ldots)$ of $T_1$, an infinite set $A$ of $|\overline{x}|$-tuples and a family $\{\overline{b_I}\in M_1^{|\overline{y}|}\;|\;I \subset A\}$ such that
$$\mathcal M_1 \models \phi(\overline{a};\overline{b_I}) \Leftrightarrow \overline{a} \in I\ \ \text{ for all }\overline{a} \in A \text{.}$$
Take an arbitrary model $\mathcal M_2=(M_2,\ldots)$ of $T_2$ and consider the standard simple product $\mathcal N_{\text{st}}$ of $\mathcal M_1$ and $\mathcal M_2$.
Let $\widetilde{\phi}(\overline{x};\overline{y})$ be the standard conversion of $\phi(\overline{x};\overline{y})$.
Take a point $c \in M_2$.
For any tuple $\overline{d}=(d_1, \ldots d_n) \in M_1^n$, the notation $\widehat{d}$ denotes the tuple $((d_1,c),\ldots,(d_n,c)) \in (M_1 \times M_2)^n$.
Set $A'=\{\widehat{a}\;|\; \overline{a} \in A\}$.
We obviously have 
$$\mathcal N_{\text{st}} \models \phi(\overline{a};\widehat{b_I}) \Leftrightarrow \overline{a} \in I\ \ \text{ for all }\overline{a}\in A' \text{.}$$
We have demonstrated that $\widetilde{\phi}(\overline{x};\overline{y})$ has IP because $\mathcal N_{\text{st}}$ is a model of $T_{\text{sim}}$ by Theorem \ref{thm:standard_sp}.
It means that $T_{\text{sim}}$ is not NIP.

We next assume that $T_{\text{sim}}$ is not NIP.
There exists an $L_{\text{sim}}$-formula having IP.
We may assume that it is of the form 
\[
\bigvee_{i=1}^N \widetilde{\psi}_{1i}(\overline{x};\overline{y}) \wedge \widetilde{\psi}_{2i}(\overline{x};\overline{y}) 
\] 
by Theorem \ref{thm:standard_sp2}.
Here, $\widetilde{\psi}_{ij}(\overline{x};\overline{y})$ is the standard conversion of an $L_i$-formula $\psi_{ij}(\overline{x};\overline{y})$ for $i=1,2$ and $1 \leq j \leq N$.
At least one of $\widetilde{\psi}_{ij}(\overline{x};\overline{y})$ has IP by \cite[Lemma 2.9]{S}.
We consider the case in which $\widetilde{\psi}_{11}(\overline{x};\overline{y})$ has IP and set $\psi(\overline{x};\overline{y})=\psi_{11}(\overline{x};\overline{y})$ and $\widetilde{\psi}(\overline{x};\overline{y})=\widetilde{\psi}_{11}(\overline{x};\overline{y})$ for simplicity.
We can prove in the same manner in the other cases.

There exists a model $\mathcal N=(N,\ldots)$ of $T_{\text{sim}}$, an infinite set $A$ of $|\overline{x}|$-tuples and a family $\{\overline{b_I}\in N^{|\overline{y}|}\;|\;I \subset A\}$ such that
$$\mathcal N \models \widetilde{\psi}(\overline{a};\overline{b_I}) \Leftrightarrow \overline{a} \in I\ \ \text{ for all }\overline{a} \in A \text{.}$$
We may assume that $\mathcal N$ is the standard simple product of a model $\mathcal M_1 =(M_1,\ldots)$ of $T_1$ and a model $\mathcal M_2 =(M_2,\ldots)$ of $T_2$ by Theorem \ref{thm:standard_sp}.
Let $\pi:N=M_1 \times M_2 \rightarrow M_1$ be the projection.
We get 
$$\mathcal M_1 \models \psi(\pi(\overline{a});\pi(\overline{b_I})) \Leftrightarrow \overline{a} \in I\ \ \text{ for all }\overline{a} \in A$$
by Lemma \ref{lem:standard}.
Take $\overline{a},\overline{b} \in A$ with $\overline{a} \neq \overline{b}$.
We can take a subset $I'$ of $A$ with $\overline{a} \in I'$ and $\overline{b} \not\in I'$.
Since we have $\mathcal M_1 \models \psi(\pi(\overline{a});\pi(\overline{b_{I'}}))$ and $\mathcal M_1 \not\models \psi(\pi(\overline{b});\pi(\overline{b_{I'}}))$, we get $\pi(\overline{a}) \neq \pi(\overline{b})$.
Set $A'=\{\pi(\overline{a})\;|\; \overline{a} \in A\}$ and, for any subset $J \subset A'$, we set $\overline{c_J}=\pi(\overline{b_I})$, where $I$ is the unique subset with $\pi(I)=J$.
We obviously have
$$\mathcal M_1 \models \psi(\overline{a};\overline{c_J}) \Leftrightarrow \overline{a} \in J\ \ \text{ for all }\overline{a} \in A'\text{.}$$
We have demonstrated that the theory $T_1$ is not NIP.
\end{proof}

We can prove a similar assertion for distal theories.
We first review indiscernible sequences.
\begin{definition}
Consider a theory $T$ of a language $L$.
Let $\mathbb M=(M,\ldots)$ be a monster model of $T$ and $A$ be a subset of $M$.
A sequence $(a_i\;|\;i \in I)$ of elements in $M$ indexed by a linearly ordered set $I$ is $A$-\textit{indiscernible} if, for any $n < \omega$, two strictly increasing $n$-tuples $i_1<\ldots<i_n$ and $j_1<\ldots<j_n$ and any $L(A)$-formula $\phi(x_1,\ldots, x_n)$,
we have 
$$\mathcal N \models \phi(a_{i_1},\ldots,a_{i_n}) \leftrightarrow \phi(a_{j_1},\ldots,a_{j_n})\text{.}$$
An $\emptyset$-indiscernible sequence is simply called \textit{indiscernible}.
\end{definition}

We get the following lemma:
\begin{lemma}\label{lem:indiscrenible}
Let $L_1$ and $L_2$ be two disjoint languages which do not have a function symbol and contains at least one constant symbol.
Let $T_1$ and $T_2$ be their theories.
Let $T_{\text{sim}}$ be the theory of simple product of $T_1$ and $T_2$.

Take a monster model $\mathcal N=(N,\ldots)$ of the theory $T_{\text{sim}}$.
The monster model $\mathcal N$ is isomorphic to the standard simple product of a model $\mathcal M_1=(M_1,\ldots)$ of $T_1$ and a model $\mathcal M_2=(M_2,\ldots)$ of $T_2$ by Theorem \ref{thm:standard_sp}.
The notation $\pi_k:N = M_1 \times M_2 \rightarrow M_k$ denotes the projection for $k=1,2$.

Take a small subset $A$ of $N$ and set $A_k=\pi_k(A)$ for $k=1,2$.
Consider sequences $(a_i\;|\; i \in I)$ and $(b_i\;|\; i \in I)$ of elements in $M_1$ and $M_2$ indexed by a linearly ordered set $I$, respectively.
The sequence $((a_i,b_i)\in N\;|\;i \in I)$ is $A$-indiscernible if and only if the sequence $(a_i\;|\; i \in I)$ is $A_1$-indiscernible and the sequence $(b_i\;|\; i \in I)$ is $A_2$-indiscernible.
\end{lemma}
\begin{proof}
We first demonstrate that the sequence $(a_i\;|\; i \in I)$ is $A_1$-indiscernible under the assumption that $((a_i,b_i)\in N\;|\;i \in I)$ is $A$-indiscernible.
We can demonstrate that the sequence $(b_i\;|\; i \in I)$ is $A_2$-indiscernible in the same manner.
Take an arbitrary $L_1(A_1)$-formula $\phi(x_1,\ldots, x_n)$.
The standard conversion $\widetilde{\phi}(x_1,\ldots, x_n)$ of $\phi(x_1,\ldots, x_n)$ is an $L_{\text{sim}}(A)$-formula by Lemma \ref{lem:standard}.
Take two strictly increasing $n$-tuples $i_1<\ldots<i_n$ and $j_1<\ldots<j_n$.
We have 
\begin{align*}
&\mathcal N\models \widetilde{\phi}((a_{i_1},b_{i_1}),\ldots,(a_{i_n},b_{i_n})) \Leftrightarrow \mathcal M_1 \models \phi(a_{i_1},\ldots a_{i_n}) \text{ and }\\ 
&\mathcal N\models \widetilde{\phi}((a_{j_1},b_{j_1}),\ldots,(a_{j_n},b_{j_n})) \Leftrightarrow \mathcal M_1  \models \phi(a_{j_1},\ldots a_{j_n})
\end{align*}
by Lemma \ref{lem:standard}.
We also have 
$$\mathcal N \models \widetilde{\phi}((a_{i_1},b_{i_1}),\ldots,(a_{i_n},b_{i_n}))  \leftrightarrow \widetilde{\phi}((a_{j_1},b_{j_1}),\ldots,(a_{j_n},b_{j_n}))$$
by the assumption.
We have demonstrated that $\mathcal M_1 \models \phi(a_{i_1},\ldots a_{i_n}) \leftrightarrow \phi(a_{j_1},\ldots a_{j_n})$.

We next demonstrate the opposite implication.
Take two strictly increasing $n$-tuples $i_1<\ldots<i_n$ and $j_1<\ldots<j_n$.
We also take an arbitrary $L(A)$-formula $\phi(x_1,\ldots, x_n)$.
We want to show that $\mathcal N \models \phi((a_{i_1},b_{i_1}),\ldots,(a_{i_n},b_{i_n})) \leftrightarrow \phi((a_{j_1},b_{j_1}),\ldots,(a_{j_n},b_{j_n}))$.
We may assume that $\phi(x_1,\ldots,x_n)$ is of the form $$\widetilde{\psi}_1(x_1,\ldots, x_n) \wedge \widetilde{\psi}_2(x_1,\ldots, x_n)\text{,}$$ 
where $\widetilde{\psi}_k(x_1,\ldots,x_n)$ is the standard conversion of a $L(A_k)$-formula $\psi_k(x_1,\ldots,x_n)$ for $k=1,2$ by Lemma \ref{lem:standard} and Theorem \ref{thm:standard_sp2}.
We obviously have
\begin{align*}
&\mathcal N \models \phi((a_{i_1},b_{i_1}),\ldots,(a_{i_n},b_{i_n}))\\
&  \Leftrightarrow \mathcal M_1 \models \psi_1(a_{i_1},\ldots, a_{i_n}) \text{ and } \mathcal M_2 \models \psi_2(b_{i_1},\ldots, b_{i_n}) \\
&  \Leftrightarrow \mathcal M_1 \models \psi_1(a_{j_1},\ldots, a_{j_n}) \text{ and } \mathcal M_2 \models \psi_2(b_{j_1},\ldots, b_{j_n})\\
&  \Leftrightarrow \mathcal N \models \phi((a_{j_1},b_{j_1}),\ldots,(a_{j_n},b_{j_n}))
\end{align*}
by Lemma \ref{lem:standard}.
We have finished the proof.
\end{proof}

We now recall the definition of a distal theory.
\begin{definition}\label{def:distal}\cite{S4}
A theory $T$ is \textit{distal} if, for any subset $A$ of the monster model, an indiscernible sequence $(a_i\;|\; i \in I)$ such that 
\begin{itemize}
\item $I =I_1+(c)+I_2$, and both $I_1$ and $I_2$ are infinite without endpoints,
\item $(a_i\;|\;i \in I_1+I_2)$ is $A$-indiscernible,
\end{itemize}
is $A$-indiscernible.
Here, the notation $J_1+J_2$ is a concatenation of $J_1$ followed by $J_2$ for linearly ordered sets $J_1$ and $J_2$.
The notation $(c)$ denotes the sequence consisting of a single element $c$. 
\end{definition} 

We get a counterpart of Theorem \ref{thm:inherit} for distal theories here.
\begin{theorem}\label{thm:inherit2}
Let $L_1$ and $L_2$ be two disjoint languages which do not have a function symbol and contains at least one constant symbol.
Let $T_1$ and $T_2$ be their theories.
Let $T_{\text{sim}}$ be the theory of simple product of $T_1$ and $T_2$.
The theory $T_{\text{sim}}$ is distal if and only if both the theories $T_1$ and $T_2$ are distal.
\end{theorem}
\begin{proof}
Immediate from Lemma \ref{lem:indiscrenible} and Definition \ref{def:distal}.
\end{proof}

\section{NIP and non-NTP$_2$ locally o-minimal theories}\label{sec:locally}
We first recall the notion of dp-minimality.
\begin{definition}\cite{DGL}
Fix a structure $\mathcal M=(M,\ldots)$.
An \text{ICT pattern} in $\mathcal M$ consists of a pair of formulas $\phi(x,\overline{y})$ and $\psi(x,\overline{y})$ and sequences $\{\overline{a_i}\;|\; i \in \omega\}$ and $\{\overline{b_i}\;|\; i \in \omega\}$ from $M$ so that, for all $i,j \in \omega$, the following is consistent:
$$\phi(x,\overline{a_i}) \wedge \psi(x,\overline{b_j}) \wedge \bigwedge_{l \neq i} \neg \phi(x,\overline{a_l}) \wedge \bigwedge_{k \neq j} \neg \psi(x,\overline{b_k})\text{.}$$

A theory is \textit{dp-minimal} if the Dp-rank $\operatorname{dp-rk}(x=x,\emptyset)=1$ for any finite tuple of variables $x$.
It is equivalent to the condition that no models of it have ICT patterns \cite[Proposition 4.22]{S}.
\end{definition}

The following proposition claims that dp-minimality classifies o-minimal structures in definably complete expansions of ordered groups.
\begin{proposition}\label{prop:dp-minimal}
A definably complete expansion of an ordered group is o-minimal if and only if it is dp-minimal 
\end{proposition}
\begin{proof}
An o-minimal theory is dp-minimal in general. 
See \cite[Theorem A.6]{S} for instance.
The opposite implication follows from \cite[Corollary 3.7]{S2} and \cite[Proposition 2.2]{M}.
\end{proof}

We finally demonstrate the theorem announced in Section \ref{sec:intro}.
We next review the definition of NTP$_2$ theories.
\begin{definition}\cite{C}\label{def:ntp2}
A formula $\phi(\overline{x};\overline{y})$ has \textit{TP$_2$} if there is an array $(b_i^t\;|\;i<\omega,t<\omega)$ of tuples of size $|\overline{y}|$ and $k<\omega$ such that:
\begin{itemize}
\item for any $\eta:\omega \rightarrow \omega$, the conjunction $\bigwedge_{t<\omega}\phi(\overline{x};b_{\eta(t)}^t)$ is consistent;
\item for any $t<\omega$, $\{\phi(\overline{x};b_i^t)\;|\;i<\omega\}$ is $k$-inconsistent.
\end{itemize}
This condition is equivalent to the following condition by the compactness theorem.
There is $k<\omega$ and, for any $m<\omega$ and $n<\omega$, there exists an array $(b_i^t\;|\;i <m,t<n)$ of tuples of size $|\overline{y}|$ such that:
\begin{itemize}
\item for any $\eta:\{t<n\} \rightarrow \{t<m\}$, the conjunction $\bigwedge_{t<n}\phi(\overline{x};b_{\eta(t)}^t)$ is consistent;
\item for any $t<n$, $\{\phi(\overline{x};b_i^t)\;|\;i<m\}$ is $k$-inconsistent.
\end{itemize}

The formula $\phi(\overline{x};\overline{y})$ is \textit{NTP$_2$} if it does not have TP$_2$.
A theory $T$ is \textit{NTP$_2$} if all formulas are \textit{NTP$_2$}.
\end{definition}

Here is the main theorem of this paper.

\begin{theorem}\label{thm:last}
There exist NIP and non-NTP$_2$ theories satisfying all the following conditions:
\begin{enumerate}
\item[(1)] It is not o-minimal;
\item[(2)] All models are strongly locally o-minimal; 
\item[(3)] It has a model which is an expansion of the linearly ordered abelian group over the reals $(\mathbb R\;;\;0,+,<)$.
\end{enumerate}
\end{theorem}
\begin{proof}
Any ordered abelian group is NIP by \cite[Theorem A.8]{S}.
In particular, the theory $T_{\mathbb Z}=\operatorname{Th}(\mathcal Z)$ of the ordered abelian group $\mathcal Z=(\mathbb Z\;;\;0,+,<)$ is NIP.
We demonstrate that there exists an expansion of $\mathcal Z$ whose theory has TP$_2$.
Let $m$ and $n$ be positive integers.
Consider the set $S_{n,m}$ of all the maps from $\{t \in \mathbb Z\;|\;1 \leq t \leq n\}$ to $\{t \in \mathbb Z\;|\;1 \leq t \leq m\}$.
The disjoint union $\mathcal S = \bigcup_{m,n}S_{n,m}$ is a countable set.
Take a bijection $\iota:\omega \rightarrow \mathcal S$.
For any $\eta \in S_{n,m}$, we set $X_{\eta}=\{(k-1)m+\eta(k)\;|\;1 \leq k \leq n\}$.

Consider the structure $(\mathbb Z\;;\;0,+,<,P)$, where $P$ is a binary predicate symbol.
We define $P$ by $$\models P(i,j) \Leftrightarrow j \in X_{\iota(i)}\text{,}$$
where $\iota(i)$ is a map from $\{t \in \mathbb Z\;|\;1 \leq t \leq n\}$ to $\{t \in \mathbb Z\;|\;1 \leq t \leq m\}$ for some $n$ and $m$.

Let $\{S_i\}_{i < \omega}$ be the indexed family of all the finite subsets of $\mathbb Z$.
Consider the structure $(\mathbb Z\;;\;0,+,<,P)$, where $P$ is a binary predicate symbol.
We define $P$ by $$\models P(i,j) \Leftrightarrow j \in S_i\text{.}$$
It is obvious that the formula $\phi(x,y):=P(x,y)$ has TP$_2$.
This structure is denoted by  $\mathcal Z'$.
In this proof, the notation $T'_{\mathbb Z}$ denotes the theory of this structure.

Consider an arbitrary o-minimal structure $\mathcal M$ over the ordered abelain group of reals.
Consider the structure $[0,1)_{\text{def}}$ whose universe is $[0,1)$ given in \cite[Definition 2]{KTTT}.
It is an o-minimal structure.
The notation $T_{\text{o-min}}$ is the theory of this structure.
The theory $T_{\text{o-min}}$ is NIP by \cite[Example 2.12]{S}.

Let $T_{\text{sim}}$ be the theory of simple product of $T_{\mathbb Z}$ and $T_{\text{o-min}}$.
Let $T'_{\text{sim}}$ be that of $T'_{\mathbb Z}$ and $T_{\text{o-min}}$.
They are the desired theories.
In fact, the former is NIP by Theorem \ref{thm:inherit}.
It is obvious that the standard conversion of the formula $\phi(x,y)=P(x,t)$ has TP$_2$.
Consider an arbitrary model of one of them.
Since it is a simple product of a discrete order and an o-minimal structure by Theorem \ref{thm:standard_sp} and Corollary \ref{cor:standard_sp}, it is strongly locally o-minimal by \cite[Theorem 19]{KTTT}.
The standard simple product of $\mathcal Z$ and $[0,1)_{\text{def}}$ is isomorphic to an expansion of linearly ordered abelian group over the reals $(\mathbb R\;;\;0,+,<)$ by \cite[Example 15.3]{KTTT}.
The standard simple product of $\mathcal Z'$ and $[0,1)_{\text{def}}$ is obviously an expansion of the above standard simple product.
Hence, both $T_{\text{sim}}$ and $T'_{\text{sim}}$ have a model which is an expansion of the linearly ordered abelian group over the reals $(\mathbb R\;;\;0,+,<)$.
These models are not o-minimal trivially because the set $\mathbb Z$ is definable in both structures. 
\end{proof}

\begin{remark}
We can construct distal and non-distal locally o-minimal theories satisfying the conditions (1) and (2) of Theorem \ref{thm:last}.
We use distal and non-distal examples in \cite{HN} and Theorem \ref{thm:inherit2} instead of Theorem \ref{thm:inherit} in the proof.
It is the same for strongly dependence.
\end{remark}

\end{document}